\documentclass[12pt, reqno]{amsart}
\usepackage{mathrsfs}
\usepackage{amsfonts,amsmath,dsfont,amssymb,amsthm,stmaryrd,bbm,color,comment}
\usepackage[hidelinks]{hyperref}
\usepackage{appendix}
\usepackage[pdftex]{graphicx}
\usepackage{bbm}
\usepackage{array}
\newcolumntype{L}{>{\centering\arraybackslash}m{2.5cm}}
\usepackage{multirow}
\usepackage{xcolor}
\usepackage[percent]{overpic}
\usepackage{algorithm}
\usepackage{algpseudocode}
\usepackage{mathtools}

\newtheorem{thm}{Theorem}[section]
\newtheorem{lem}[thm]{Lemma}

\theoremstyle{definition}
\newtheorem{defn}[thm]{Definition}

\newtheorem{conj}[thm]{Conjecture}

\theoremstyle{remark}
\newtheorem{rmk}[thm]{Remark}
\numberwithin{equation}{section}

\newcommand{\supp}{\mathrm{supp}}

\newcommand{\Var}{\mathrm{Var}}

\newcommand{\R}{\mathbb{R}}
\newcommand{\N}{\mathbb{N}}
\newcommand{\Z}{\mathbb{Z}}

\newcommand{\E}{\mathbb{E}}

\def\P{{\mathbb P}}

\begin{document}
\setcounter{page}{1}

\color{black}{

\centerline{}

\centerline{}

\title[R\'enyi entropies of weighted Bernoulli sums]{Multiplicative comparisons of R\'enyi entropies for weighted Bernoulli sums}

\author[Jiange Li]{Jiange Li} 







\begin{abstract}

We establish multiplicative comparisons between Rényi entropies of different orders for weighted sums of independent Bernoulli random variables. In particular, we prove a logarithmic comparison between the zeroth- and infinity-order Rényi entropies, which yields a polynomial improvement over the square-root bound of Jain, Sah, and Sawhney. As an application, this leads to an improved parameterized running time for the randomized bin-packing algorithm of Nederlof, Pawlewicz, Swennenhuis, and W\c{e}grzycki. We also obtain explicit dimension-free, constant-factor comparisons between Rényi entropies of positive orders.
\end{abstract} 

\maketitle

\section{Introduction}

\subsection{Background}

Let $X$ be a discrete random variable with probability mass function  $\{f(x)\}_{x\in A}$. For $\alpha\in (0, 1)\cup (1, \infty)$, the \textit{R\'enyi entropy} $H_\alpha(X)$ of order $\alpha$ is defined by
$$
H_\alpha(X)=\frac{1}{1-\alpha}\log\sum_{x\in A}f(x)^\alpha.
$$
By taking limits, one extends this definition to  $\alpha\in\{0, 1, \infty\}$ as  $H_0(X) =\log |\supp(f)|$, $H_1(X)=-\sum_{x\in A}f(x)\log f(x)$, and $H_\infty(X)=-\log \|f\|_\infty$.
Here, $|\supp(f)|$ denotes the cardinality of the support of $f$, $\|f\|_\infty=\sup_{x\in A}f(x)$, and $H_1(X)$ is the classical Shannon entropy, where the subscript one is usually omitted. 

Jensen's inequality implies that $H_\alpha(X)$ is non-increasing as a function of $\alpha\in [0, \infty]$.
A reverse H\"{o}lder type question is when can this monotonicity be reversed. The norm inequality $\|f\|_\alpha\geq \|f\|_\beta$ for $0\leq \alpha<\beta\leq \infty$ yields
\begin{align}\label{eq:multiplicative-compr}
H_\alpha(X)\leq \frac{1-1/\beta}{1-1/\alpha} \cdot H_\beta(X), \quad 1<\alpha<\beta\leq \infty.
\end{align}
For the range $0 \le \alpha < \beta \le 1$, however, no universal multiplicative constant exists. To see this, consider a Bernoulli random variable $X\sim\text{Bern}(p)$ with $0<p<1$. As $p\to0$, one can check that
$$
H_0(X)=\log 2, \qquad H_\alpha(X)\sim \frac{p^\alpha}{1-\alpha}\quad (0<\alpha<1), \qquad H_1(X)\sim p\log (1/p).
$$

This raises the question: For which classes of distributions does the multiplicative comparison in \eqref{eq:multiplicative-compr} remain valid over the entire range $0 \le \alpha < \beta \le \infty$? In the endpoint case $\alpha=0, \beta=\infty$, such an inequality for a finitely supported random variable would assert that
$$
|\supp(f)|\le \frac{1}{\|f\|_\infty^c}
$$
for some absolute constant $c>1$. This is a manifestation of the anti-concentration property of $X$. An important class of distributions whose anti-concentration properties have been extensively studied is that of weighted sums of independent Bernoulli random variables.  The following conjecture was communicated to us by Mokshay Madiman; an equivalent formulation appears in \cite{JSS21} (see equation (1.2) therein).


\begin{conj} \label{conj:main}
Let $X$ be a random vector uniformly distributed on $\{0, 1\}^n$. For a weight vector $w\in \R^n$, we define $S_w=\sum_{i=1}^nw_iX_i$. Then, for all $w\in\mathbb{R}^n$, it holds that
\begin{equation}\label{eq:conj-JSS}
H_0(S_w)\leq 2 H_\infty(S_w).
\end{equation}
\end{conj}

\begin{rmk}
This problem originates in the work of Nederlof, Pawlewicz, Swennenhuis, and W\c{e}grzycki on fast algorithms for bin packing \cite{NPSW23}.  Jain, Sah,
and Sawhney \cite{JSS21} obtained the following equivalent bound
\begin{equation}\label{eq:JSS21}
H_0(S_w)\le C\sqrt{nH_\infty (S_w)},
\end{equation}
where $C$ is an absolute constant, and this gives a polynomial improvement over an earlier estimate in \cite{NPSW23}. It was observed in \cite{JSS21} that such a conjecture, if true, is sharp. One can consider the weight vector $w=(c_1, \cdots, c_1, c_2, \cdots, c_2, \cdots, c_{n/k}, \cdots, c_{n/k})$, where each $c_i$ appears $k$ times and $c_i$ is sufficiently small compared to $c_{i+1}$ for all $i=1, \cdots, n/k$. Note the block sums are independent binomial random variables $\text{Bin}(k, 1/2)$ and their sums are all distinct. Consequently, we have
$$
H_0(S_w)=\frac{n}{k}\log (k+1), \quad H_\infty(S_w)=\frac{n}{k}\left(
k\log 2-\log{k\choose \lceil k/2\rceil}
\right).
$$
Stirling's formula shows that $H_0(S_w)/H_\infty(S_w)$ tends to 2 as $k\to\infty$. Therefore the proposed constant in inequality \eqref{eq:conj-JSS} cannot be improved.
\end{rmk}

\begin{rmk}
One may also consider additive comparisons of R\'enyi entropies for other classes of distributions.  In this direction, Melbourne and Palafox-Castillo \cite{MPC23} showed that, for integer-valued monotone log-concave random variables and all $0\leq \alpha<\beta\leq\infty$,
$$
H_\alpha(X)-H_\beta(X)\le H_\alpha(Z)-H_\beta(Z),
$$
where $Z$ is a geometric random variable. This is the discrete analogue of Theorem 2.9 of Fradelizi, Madiman, and Wang \cite{FMW16}. (An earlier partial result was obtained by Bobkov and Madiman \cite{BM11}, and an extension was later obtained by Fradelizi, Li, and Madiman \cite{FLM20}).
\end{rmk}

\subsection{Main results}

Our first result gives a dimension-free, constant-factor bound relating R\'enyi entropies of positive orders. 

\begin{thm}\label{thm:non-zero-order}
Let $\alpha>0$. Set $k_\alpha=\lceil \log_2 (1/2\alpha)\rceil$ for $\alpha\in (0, 1/2)$. Under notations of Conjecture \ref{conj:main}, for all $w\in \mathbb{R}^n$, we have
$$
H_\alpha(S_w)\le C_\alpha H_\infty(S_w),
$$
where the constant $C_\alpha$ is given by
$$
C_\alpha=
\begin{cases}
\frac{k_\alpha+2(1-2^{k_\alpha}\alpha)}{1-\alpha},  &\alpha\in(0, 1/2),\\
2, &\alpha\in [1/2, 2),\\
\frac{\alpha}{\alpha-1}, &\alpha\in [2, \infty).
\end{cases}
$$
\end{thm}

Our second result provides a logarithmic bound between the zeroth- and infinity-order R\'enyi
entropies, which yields a polynomial improvement over inequality \eqref{eq:JSS21}. 

\begin{thm}\label{thm:zero-order}
For $w\in\mathbb{R}^n$, denote $m=|\{i\in [n]: w_i\ne 0\}|$. Under notations of Conjecture \ref{conj:main}, we have
$$
H_0(S_w)
\leq C(w) H_\infty(S_w),
$$
where $C(w)$ is given by
$$
C(w)=\min\left\{\log_2(m+1), 2+\log_2\left(\frac{m\log 2}{H_\infty(S_w)}\right)\right\}.
$$
\end{thm}

\begin{rmk}
For $w\in\mathbb{R}^n$, we write 
\begin{equation}\label{eq:rho-r}
\rho(w)=\max_{t\in\R}\P(S_w=t), \quad
\mathcal{R}(w)=\{\langle x, w\rangle: x\in \{0, 1\}^n\}.
\end{equation}
Here, $\langle\cdot, \cdot\rangle$ denotes the inner product. In the notations of \cite{NPSW23, JSS21}, we have
$$
\rho(w)\ge e^{-\varepsilon n} 
\Longrightarrow |\mathcal{R}(w)|\le e^{\delta(\varepsilon)n},
$$
where
\begin{equation}\label{eq:delta}
\delta(\varepsilon)=O\left(\varepsilon\log\frac1\varepsilon\right).
\end{equation}
This improves the quadratic relation $\delta(\varepsilon)=O(\sqrt{\varepsilon})$ of Jain, Sah, and Sawhney \cite{JSS21}, and is optimal up to a logarithmic factor. Indeed, if $\rho(w)\ge e^{-\varepsilon n}$, then $H_\infty(S_w)\le \varepsilon n$, and Theorem \ref{thm:zero-order} gives
$$
\log |\mathcal{R}(w)|
=H_0(S_w)
\le \left(2+\log_2\left(\frac{n\log 2}{H_\infty(S_w)}\right)\right)H_\infty(S_w)
=O\left(\varepsilon\log\frac1\varepsilon\right)n.
$$
\end{rmk}

\begin{rmk} [Application to bin packing]
Bin packing is a classical NP-complete problem, stated as follows: Given $n$ items with weights $w_1,\dots,w_n$ and $m$ bins with capacities $c_1,\dots,c_m$, is there a way to assign the items to the bins without violating the capacity constraints? Bj\"{o}rklund, Husfeldt, and Koivisto \cite{BHK09} provided an algorithm for solving bin packing in time $\tilde{O}(2^n)$, where $\tilde{O}$ omits polynomial factors in $n$. Recently, Nederlof, Pawlewicz, Swennenhuis, and W\c{e}grzycki \cite{NPSW23} showed that for every $m\in\N$ there is a constant $\sigma_m>0$ such that the instance of bin packing with $m$ bins can be solved by a randomized algorithm in time $\tilde{O}(2^{(1-\sigma(m))n})$ with high probability. Their analysis crucially relies on the dependence of $\delta(\varepsilon)$ on $\varepsilon$, and the value of $\sigma(m)$ satisfies
\[
\sigma(m)\le 2^{-m^9}.
\]
Using the quadratic relation $\delta(\varepsilon)=O(\sqrt{\varepsilon})$, Jain, Sah, and Sawhney \cite{JSS21} improved this bound to $\sigma(m)=\tilde{\Omega}(m^{-12})$, where $\tilde{\Omega}$ hides logarithmic factors in $m$. Following the same black-box argument as Jain, Sah, and Sawhney, we can apply the nearly optimal relation \eqref{eq:delta} to further improve the bound to $\sigma(m)=\tilde{\Omega}(m^{-6})$.
\end{rmk}


\subsection{Comparison with related work} 
Recall $\rho(w)$ and $\mathcal{R}(w)$ defined in \eqref{eq:rho-r}. For each $t\in \mathcal{R}(w)$, we select one $x\in \{0, 1\}^n$ such that $\langle w, x\rangle=t$, and denote by $A$ the collection of such vectors. Clearly, $|A|=|\mathcal{R}(w)|$ and  
$\langle w, x\rangle\neq \langle w, x'\rangle$ for distinct $x, x'\in A$. We select $\tau\in \mathcal{R}(w)$ such that $\rho(w)=\P(S_w=\tau)$ and write $B=\{x\in \{0, 1\}^n: \langle w, x\rangle=\tau\}$. Then $H_0(S_w)=\log |A|$, $H_\infty(S_w)=\log (2^n/|B|)$, and Conjecture \ref{conj:main} asserts that
$$
|A|\le \left(\frac{2^n}{|B|}\right)^2.
$$
Both \cite{NPSW23} and \cite{JSS21} couple $A$ and $B$ via $A+k\cdot B$, where $k\cdot B:=B+\cdots+B$ denotes the $k$-fold Minkowski sum. A simple but crucial observation is the injectivity of the map
$$
A\times k\cdot B\to A+k\cdot B,\quad (a, b)\mapsto a+b,
$$
whcih gives
$$
|A|=\frac{|A+k\cdot B|}{|k\cdot B|}.
$$
\begin{enumerate}
\item Instead of counting all $A+k\cdot B$, \cite{NPSW23} introduces the set of ballanced pairs $P\subset A\times k\cdot B$, where $(a, b)\in P$ means that $b$ is approximately binomial relative to the support of $a$. The upper bound of $|P|$ relies on the construction of a large set of balanced vectors in $k\cdot B$ (i.e., approximately binomial vectors). The lower bound of $|P|$ crucially relies on injectivity of the aforementioned map and that $a+b$ is approximately binomial for $(a, b)\in P$. 
\item The argument of \cite{JSS21} is more probabilistic. Let $\nu_B$ and $\nu$ denote the uniform measures on $B$ and $\{0, 1\}^n$, respectively. The injectivity of the aforementioned map gives
$$
|A|\le \sum_{x\in \{0, 1, \cdots, k+1\}^n} \nu_B^{\ast k}(x-a(x)),
$$
where $\nu_B^{\ast k}$ is the $k$-fold self-convolution of $\nu_B$. In \cite{JSS21}, the distribution $\nu_B^{\ast k}$ is approximated by $\nu^{\ast k}$ at the cost of a multiplicative factor $(2^n/|B|)^k$ and it gives
$$
|A|\le \left(\frac{2^n}{|B|}\right)^k\sum_{x\in \{0, 1, \cdots, k+1\}^n} \nu^{\ast k}(x-a(x)).
$$
Note $\nu^{\ast k}$ is exactly the product of Binomial distributions $\text{Bin}(k, 1/2)^{\otimes n}$. Thus the remaining difficulty is to bound the ratio between $\nu^{\ast k}(x)$ and its translate $\nu^{\ast k}(x-a(x))$. The multiplicative factor $(2^n/|B|)^k$ is not sharp if $B$ has certain expansion property, and it can likely be improved to be sub-exponential in $k$. 
\end{enumerate}

The key arguments in our proofs are summarized below.
\begin{enumerate}
\item \textbf{Theorem \ref{thm:non-zero-order}}. Let $X$ and $Y$ be random vectors, with $X$ uniform on $\{0, 1\}^n$ and $Y$ uniform on the largest fiber $B$. We couple $X$ and $Y$ by setting $Z=X\oplus Y$, where $\oplus$ denotes coordinatewise addition modulo two. This coupling has two useful properties: $Z$ is uniform on $\{0, 1\}^n$, independent of $Y$; and, conditioning on $Z$, the weighted sum $S_w$ can be expressed in terms of $Y$. Moreover, the conditional distribution of $S_w$ given the value of $Z$ factors into a product involving two independent partial sums. Applying the Cauchy–Schwarz inequality to this product doubles the order of the Rényi entropy. Repeating this dyadic induction reaches the comparatively easy base case $\alpha=1/2$.

\item \textbf{Theorem \ref{thm:zero-order}}. This result combines bounds from Theorems \ref{thm:zero-order-a} and \ref{thm:zero-order-b}. For a weight vector $w=(w_1, \cdots, w_n)$, the key idea is to choose a maximal dissociated set of weights $\{w_i, i\in J\}$, whose subset sums are all distinct. By maximality, every remaining weight can be expressed as a signed combination of these core weights. Consequently, the entire weighted sum is determined by integer coefficients attached to the core weights, each taking at most $m+1$ possible values, where $m$ is the number of nonzero weights of $w$. This bounds the number of possible values of $S_w$. Meanwhile, this maximal dissociated set alone produces $2^{|J|}$ distinct, equally likely sums. These observations yield Theorem \ref{thm:zero-order-b} as follows
$$
H_0(S_w)\leq \log_2(m+1) H_\infty(S_w).
$$
The simple observation that every nonzero probability of $S_w$ is at least $2^{-m}$ gives the interpolation inequality \eqref{eq:H0-interpolation} as follows
$$
H_0(S_w)\leq(1-\alpha)H_\alpha(S_w)+\alpha m\log 2.
$$
Then we select $\alpha$ according to the size of $H_\infty(S_w)$. Invoking Theorem \ref{thm:non-zero-order}, we can control $H_\alpha(S_w)$ by $H_\infty(S_w)$. With this choice of $\alpha$, this gives
$$
H_0(S_w)\leq \left(2+\log_2\left(\frac{m\log 2}{H_\infty(S_w)}\right)\right) H_\infty(S_w).
$$
If $H_\infty(S_w)\ge (m\log 2)/2$, the elementary bound $H_0(S_w)\leq m\log 2$ suffices. This proves Theorem \ref{thm:zero-order-a}.
\end{enumerate}

\section{Proof of Theorem \ref{thm:non-zero-order}}

Let $X$ be a random vector uniform on $\{0, 1\}^n$. Recall $S_w=\sum_{i=1}^nw_iX_i$ and set
\begin{equation}\label{eq:f_w-rho_w}
f_w(t)=\P(S_w=t), \quad \rho(w)=\max_{t\in \R} f_w(t).
\end{equation}
Choose $\tau\in\R$ such that $f_w(\tau)=\rho(w)$ and define the fiber
\begin{equation}\label{eq:largest fiber}
B=\left\{x\in \{0, 1\}^n: \langle w, x\rangle=\tau\right\}.
\end{equation}
Clearly, $|B|=2^n\rho(w)$. Let $Y$ be a random vector uniform on $B$, independent of $X$, and 
$$
Z=X\oplus Y,
$$
where $\oplus$ denotes coordinatewise addition modulo two.  For $z\in\{0, 1\}^n$, we define
\begin{equation}\label{eq:cond-pmf S_w}
f_w(t|z)=\P(S_w=t|Z=z).
\end{equation}
Whenever $f_w(t)>0$, we define the normalized likelihood function
\begin{equation}\label{eq:likelihood}
\phi(t|z)=\frac{f_w(t|z)}{f_w(t)}.
\end{equation}

\begin{lem} [Largest fiber coupling]\label{lem:largest fiber coupling}
The following properties hold.
\begin{enumerate}
\item The random vector $Z$ is uniformly distributed on $\{0, 1\}^n$ and independent of $Y$.
\item We have
$$
0\le \phi(t|z)\le \frac{1}{\rho(w)}, \qquad 2^{-n}\sum_{z\in\{0, 1\}^n}\phi(t|z)=1.
$$
\item Given $z\in\{0, 1\}^n$, we let $I=\{i\in [n]: z_i=0\}$ and set
$$
U_z=\sum_{i\in I}w_iX_i, \qquad V_z=\sum_{i\notin I}w_iX_i, \qquad c_z=\sum_{i\notin I}w_i.
$$
Then $U_z, V_z$ are independent, $U_z+V_z=S_w$, and $f_w(t|z)$ is the pushforward of the conditional distribution of $U_z$ given $U_z+V_z=\tau$ under the map $u\mapsto c_z-\tau+2u$.
\end{enumerate}
\end{lem}

\begin{proof}
(1) For all $z\in \{0, 1\}^n$ and $y\in B$, we have
$$
\P(Z=z, Y=y)=\P(X=z\oplus y)\cdot\P(Y=y)= 2^{-n}\cdot |B|^{-1}.
$$
Summing over $y$ proves uniformity of $Z$ and then the independence between $Y$ and $Z$.

(2) Write $Q=\{x\in \{0, 1\}^n: \langle w, x\rangle=t\}$. Clearly, $\{S_w=t\}=\{X\in Q\}$. Recall definitions in \eqref{eq:likelihood}, \eqref{eq:cond-pmf S_w} and \eqref{eq:f_w-rho_w}. We have
\begin{align*}
\phi(t|z) 
&=\frac{\sum_{x\in Q}\P(X=x|Z=z)}{\P(S_w=t)}\\
&=\frac{\sum_{x\in Q}\P(Y=x\oplus z)}{|Q|\cdot 2^{-n}}\\
&\le \frac{|Q|\cdot |B|^{-1}}{|Q|\cdot 2^{-n}}\\
&=\frac{1}{\rho(w)}.
\end{align*}

The other identity follows from the total law of probability 
$$
\P(S_w=t)=2^{-n}\sum_{z\in\{0, 1\}^n}\P(S_w=t|Z=z).
$$

(3) The vector $Y$ is independent of $Z$ and it remains uniform on $B$ given $Z=z$. Then
\begin{align*}
S_w &=\sum_{i=1}^nw_i(Y_i\oplus z_i)\\
&=\sum_{i\in I}w_iY_i+\sum_{i\notin I}w_i(1-Y_i)\\
&=\sum_{i\notin I}w_i-\tau+2\sum_{i\in I}w_iY_i.
\end{align*}
The uniformity of $Y$ on $B$ is exactly the conditional distribution of $X$ given $S_w=\tau$. Therefore, we have
$$
S_w\overset{d}{=}c_z-\tau+2U_z.
$$
This proves the pushforward statement.
\end{proof}

\begin{proof} [Proof of Theorem \ref{thm:non-zero-order}]
Let $0<\gamma<1$. Recall $\phi(t|z)$ defined in \eqref{eq:likelihood} and statement (2) of Lemma \ref{lem:largest fiber coupling}. If $\phi(t|z)>0$, we have
\begin{align*}
f_w(t|z)^\gamma 
&=f_w(t)^\gamma\phi(t|z)\phi(t|z)^{\gamma-1}\\
&\ge f_w(t)^\gamma\phi(t|z)\rho(w)^{1-\gamma}.
\end{align*}
This inequality remains valid if $\phi(t|z)=0$. Then we take a uniform average over $z$, use the identity $2^{-n}\sum_z\phi(t|z)=1$ in Lemma \ref{lem:largest fiber coupling}, and sum over $t$ to obtain
\begin{equation}\label{eq:moment-f}
\sum_tf_w(t)^\gamma\le \rho(w)^{\gamma-1}\cdot 2^{-n}\sum_{z} \sum_tf_w(t|z)^\gamma.
\end{equation}
For any given $z\in\{0, 1\}^n$, we know from statement (3) of Lemma \ref{lem:largest fiber coupling} that $f_w(t|z)$ and  the conditional distribution of $U_z$ given $S_w=U_z+V_z=\tau$ have the same $\gamma$-th moment. For any $u\in\R$, we have
$$
\P(U_z=u|S_w=\tau)=\frac{\P(U_z=u)\cdot \P(V_z=\tau-u)}{\rho(w)}.
$$
Therefore, we have
\begin{align}\label{eq:moment-f_z}
\sum_tf_w(t|z)^\gamma &=\rho(w)^{-\gamma}\sum_u \P(U_z=u)^\gamma \cdot\P(V_z=\tau-u)^\gamma \notag\\
&\le \rho(w)^{-\gamma}\left[\sum_u \P(U_z=u)^{2\gamma}\right]^{1/2} \left[\sum_{v}\P(V_z=v)^{2\gamma}\right]^{1/2}.
\end{align}
For $\gamma=1/2$, we combine \eqref{eq:moment-f} and \eqref{eq:moment-f_z} to obtain
$$
\sum_tf_w(t)^{1/2}\le \rho(w)^{-1}\quad \text{i.e.,}\quad H_{1/2}(S_w)\le 2H_\infty (S_w).
$$
Then the monotonicity of R\'enyi entropy gives 
\begin{equation}\label{eq:H_1/2}
H_{\gamma}(S_w)\le 2H_\infty (S_w), \quad \gamma\ge 1/2.
\end{equation}
This, together with \eqref{eq:multiplicative-compr}, proves the $\alpha\ge 1/2$ case of Theorem \ref{thm:non-zero-order}.

Let $D_\gamma$ denote the smallest possible dimension-free constant in Theorem \ref{thm:non-zero-order}. If the statement has been proved at order $2\gamma$ in all dimensions, 
then we can apply this hypothesis to the partial sums $U_z$ and $V_z$ and obtain
$$
\sum_u \P(U_z=u)^{2\gamma}\le e^{(1-2\gamma)D_{2\gamma}H_\infty(U_z)} \le e^{(1-2\gamma)D_{2\gamma}H_\infty(S_w)}.
$$
A similar estimate holds for $V_z$. Combined with \eqref{eq:moment-f_z} and \eqref{eq:moment-f}, we obtain
$$
\sum_tf_w(t)^\gamma\le e^{[1+(1-2\gamma)D_{2\gamma}]H_\infty(S_w)},
$$
that is
$$
H_\gamma(S_w)\le \frac{1+(1-2\gamma)D_{2\gamma}}{1-\gamma} \cdot H_\infty(S_w).
$$
Therefore, we have
$$
D_\gamma\le \frac{1+(1-2\gamma)D_{2\gamma}}{1-\gamma}.
$$
Set $\tilde{D}_\gamma=(1-\gamma)D_\gamma$ and obtain $\tilde{D}_\gamma\le 1+\tilde{D}_{2\gamma}$.
For a prescribed $\alpha\in (0, 1/2)$, we repeat this doubling argument $k_\alpha=\left\lceil \log_2\left( 1/2\alpha\right)\right\rceil$ times so that $2^{k_\alpha}\alpha\in[1/2, 1)$. Then \eqref{eq:H_1/2} gives
$$
\tilde{D}_{2^{k_\alpha}\alpha}=(1-2^{k_\alpha}\alpha)D_{2^{k_\alpha}\alpha}\le 2(1-2^{k_\alpha}\alpha).
$$
After $k_\alpha$ iterations of  $\tilde{D}_\gamma\le 1+\tilde{D}_{2\gamma}$, we obtain
$$
\tilde{D}_\alpha\le k_\alpha+\tilde{D}_{2^{k_\alpha}\alpha}\le k_\alpha+2(1-2^{k_\alpha}\alpha),
$$
which gives
$$
D_\alpha\le \frac{k_\alpha+2(1-2^{k_\alpha}\alpha)}{1-\alpha}.
$$
This proves the $0<\alpha< 1/2$ case of Theorem \ref{thm:non-zero-order}.
\end{proof}

\begin{rmk}
The proof crucially relies on the conditional factorization of the distribution of $S_w$ given the value of $Z$ as a product involving two independent partial sums. This factorization fails when $X$ is not uniform on $\{0, 1\}^n$, and therefore, the approach does not directly apply to biased Bernoulli random variables. Finding the right biased analogue of Lemma \ref{lem:largest fiber coupling} is a natural problem.
\end{rmk}


\section{Proof of Theorem \ref{thm:zero-order}}

Theorem \ref{thm:zero-order} combines bounds from Theorems \ref{thm:zero-order-a} and \ref{thm:zero-order-b}.  

\begin{thm}\label{thm:zero-order-a}
For $w\in\mathbb{R}^n$, denote $m=|\{i\in [n]: w_i\ne 0\}|$. Under notations of Conjecture \ref{conj:main}, we have
\begin{equation}
 H_0(S_w)
 \leq\left(2+\log_2\left(\frac{m\log 2}{H_\infty(S_w)}\right)\right) H_\infty(S_w).
 \label{eq:log-loss}
\end{equation}
\end{thm}

\begin{proof}
If $w$ is the zero vector, there is nothing to prove. Therefore, we assume that $w$ is nonzero. If $H_\infty(S_w)\ge(m\log 2)/2$, then the trivial bound $H_0(S_w)\leq m\log 2$ gives
$$
H_0(S_w)\le 2H_\infty(S_w).
$$
If $H_\infty(S_w)<(m\log 2)/2$, we set $\alpha=H_\infty(S_w)/m\log 2<1/2$. Note every non-zero probability satisfies $f_w(t)\ge 2^{-m}$. Hence, we have
$$
\sum_tf_w(t)^\alpha\geq |\supp(f_w)|\cdot 2^{-m\alpha},
$$
which yields
\begin{equation}\label{eq:H0-interpolation}
H_0(S_w)\leq(1-\alpha)H_\alpha(S_w)+\alpha m\log 2.
\end{equation}
Since $\alpha<1/2$, Theorem \ref{thm:non-zero-order} gives
\begin{align*}
(1-\alpha)H_\alpha(S_w) &\leq [k_\alpha+2(1-2^{k_\alpha}\alpha)]H_\infty(S_w)\\
&\leq(k_\alpha+1)H_\infty(S_w),
\end{align*}
where the second inequality follows from the definition of $k_\alpha=\lceil \log_2 (1/2\alpha)\rceil$ so that $2^{k_\alpha}\alpha\in[1/2, 1)$. Our selection of $\alpha$ satisfies $\alpha m\log 2=H_\infty(S_w)$. By \eqref{eq:H0-interpolation}, we have
$$
H_0(S_w)\leq(k_\alpha+2)H_\infty(S_w).
$$
Using the definitions of $\alpha=H_\infty(S_w)/m\log 2$, one can check that
$$
k_\alpha=\left\lceil\log_2\left(\frac1{2\alpha}\right)\right\rceil\leq\log_2\left(\frac{m\log 2}{H_\infty(S_w)}\right).
$$
This completes the proof of inequality \eqref{eq:log-loss}. 
\end{proof}


\begin{defn}\label{def:dissociativity}
Let $\Lambda\subseteq\R$ be a finite multiset. A subset $\Lambda'=\{\lambda_i: i=1, \cdots, |\Lambda'|\}\subseteq\Lambda$  is called \textit{dissociated} if the equation
$$
\sum_{i=1}^{|\Lambda'|}\varepsilon_i\lambda_i=0, \quad \varepsilon_j\in \{-1, 0, 1\}
$$
has only the trivial solution $\varepsilon_i=0$ for $1\le i\le |\Lambda'|$. In other words, the subset sums $\{S_{\Lambda''}: \Lambda''\subseteq\Lambda'\}$ are all distinct, where
$$
S_{\Lambda''}=\sum_{\lambda\in\Lambda''} \lambda.
$$ 
We call $\Lambda'\subseteq \Lambda$ a \textit{maximal dissociated subset} if $\Lambda'\cup \{\lambda\}$ is not dissociated for any $\lambda\in \Lambda\setminus\Lambda'$.
\end{defn} 

\begin{rmk}
The above definition of dissociated sets is slightly different from that is commonly used in additive combinatorics and harmonic analysis, where dissociativity is defined for sets (in general abelian groups) instead of multisets (for example, see \cite{Gre04} and Definition 4.32 in \cite{TV10}). 
\end{rmk}

\begin{lem}\label{lem:dissociated rep}
Let $\Lambda'=\{\lambda_i: i=1, \cdots, |\Lambda'|\}\subseteq \Lambda$ be a maximal dissociated subset. Then every $\lambda\in \Lambda$ has the representation
$$
\lambda=\sum_{i=1}^{|\Lambda'|}\varepsilon_i\lambda_i, \quad \varepsilon_i\in \{-1, 0, 1\}.
$$
\end{lem}

\begin{proof}
The statement clearly holds if $\lambda\in \Lambda'$. For every $\lambda\notin \Lambda'$, by the definition of maximal dissociativity,  $\Lambda'\cup \{\lambda\}$ is not a dissociateed set. Therefore, the equation
$$
\varepsilon \lambda+\sum_{i=1}^{|\Lambda'|}\varepsilon_i\lambda_i=0
$$
has a non-zero solution in $\{-1, 0, 1\}^{|\Lambda'|+1}$, and moreover, we have $\varepsilon\in \{-1, 1\}$. This gives the desired representation. 
\end{proof}

\begin{thm}\label{thm:zero-order-b}
For $w\in\mathbb{R}^n$, denote $m=|\{i\in [n]: w_i\ne 0\}|$. Under notations of Conjecture \ref{conj:main}, we have
\begin{equation}\label{eq:zero-order-a}
H_0(S_w)\le \log_2(m+1)H_\infty(S_w).
\end{equation}
\end{thm}

\begin{proof}
If $w=0$, there is nothing to prove. Therefore we assume that $w\ne 0$ and that $\{w_1, \cdots, w_m\}$ is the set of all nonzero weights. Let $\{w_j: j\in J\}$ be a maximal dissociated subset of $\{w_1, \cdots, w_m\}$. 
By Lemma \ref{lem:dissociated rep}, for every $i\in I=[m]\backslash J$, we have the representation
$$
w_i=\sum_{j\in J}\varepsilon_{i, j}w_j, \quad \varepsilon_{i, j}\in \{-1, 0, 1\}.
$$
Then we have
\begin{align*}
S_w &=\sum_{j\in J} w_jX_j+\sum_{i\in I}w_iX_i\\
&=\sum_{j\in J} w_jX_j+\sum_{i\in I}\left[\sum_{j\in J}\varepsilon_{i, j}w_j\right]X_i\\
&=\sum_{j\in J} w_jX_j+\sum_{j\in J}w_j\left[\sum_{i\in I}\varepsilon_{i, j}X_i\right]\\
&=\sum_{j\in J} w_j\left[X_j+\sum_{i\in I}\varepsilon_{i, j}X_i\right]\\
&=\sum_{j\in J} w_j\sum_{i\in [m]}\tilde{\varepsilon}_{i, j}X_i.
\end{align*}
Here, $\tilde{\varepsilon}_{i, j}\in \{-1, 0, 1\}$. Since $X_i\in\{0, 1\}$,  for each $j\in J$, the integer-valued random variable $\sum_{i\in [m]}\tilde{\varepsilon}_{i, j}X_i$ can take at most $m+1$ distinct values. This implies that $S_w$ can take at most $(m+1)^{|J|}$ possible values. Therefore, we have
\begin{equation}\label{eq:h_0 upper}
H_0(S_w)\le |J|\log(m+1).
\end{equation}
Since $\{w_j: j\in J\}$ is dissociated,  $\sum_{j\in J} w_jX_j$ takes $2^{|J|}$ distinct values with equal probability. Hence, we obtain
\begin{equation*}
H_\infty(S_w)\ge H_\infty\left(\sum_{j\in J} w_jX_j\right)= |J|\log 2. 
\end{equation*}
This, together with \eqref{eq:h_0 upper}, give inequality \eqref{eq:zero-order-a}.
\end{proof}


The maximal dissociated subset argument used in the proof of Theorem \ref{thm:zero-order-b} does not depend on the arithmetic structure of $\R$, and it remains effective in the context of general abelian groups. Therefore, we have the following extension of Theorem \ref{thm:zero-order-b}.

\begin{thm}\label{thm:zero-order-abelian}
Let $G$ be a general abelian group. Let $\{X_i\}_{1\le i\le n}$ be independent Bernoulli random variables. Write $p_i=\P(X_i=1)$. Suppose $p\le p_i\le 1-p$ for some $0<p\le 1/2$ and for all $1\le i\le n$.  For $g=(g_1, \cdots, g_n)\in G^n$, we define $S_g=\sum_{i=1}^ng_iX_i$. Denote $m=|\{i\in [n]: g_i\ne 0\}|$. Then we have
\begin{equation}\label{eq:extension}
H_0(S_g)\le \frac{\log(m+1)}{-\log (1-p)}H_\infty(S_g).
\end{equation}
\end{thm}

\begin{proof}
Replace $\R$ by $G$, and define dissociativity of a multiset of $G$ as in Definition \ref{def:dissociativity}. One can check that Lemma \ref{lem:dissociated rep} remains valid. Then we can proceed the proof as that of Theorem \ref{thm:zero-order-b} essentially verbatim. Inequality \eqref{eq:h_0 upper} still holds since it does not depend on the probabilities $\{p_i\}_{1\le i\le n}$. The only difference is that the random variable $\sum_{j\in J} g_jX_j$ takes $2^{|J|}$ distinct values and
$$
\max_{t}\P\left(\sum_{j\in J} g_jX_j=t\right)\le \prod_{j\in J}\max\{p_j, 1-p_j\}\le (1-p)^{|J|},
$$
which gives
\begin{equation*}
H_\infty(S_g)\ge H_\infty\left(\sum_{j\in J} g_jX_j\right)\ge-|J|\log (1-p). 
\end{equation*}
This, combined with \eqref{eq:h_0 upper}, gives the extension \eqref{eq:extension}.
\end{proof}


\section{Further directions}

Beyond Conjecture \ref{conj:main}, we believe that the following questions are also worthy of further investigation.

\begin{enumerate}
\item \textbf{Biased Bernoullis.} Theorem \ref{thm:zero-order-b} extends readily to sums of independent, non-identically distributed Bernoulli random variables, even with weights drawn from general abelian groups. It is therefore natural to seek analogues of Theorems \ref{thm:non-zero-order} and \ref{thm:zero-order-a} in this setting. One expects dimension-free bounds, that depend only on lower and upper bounds for the Bernoulli parameters. The main obstacle is to find an
appropriate biased analogue of the largest-fiber coupling in Lemma \ref{lem:largest fiber coupling}.

\item \textbf{$q$-ary random variables.} Let $\{X_i\}_{1\le i\le n}$ be independent random variables, that are uniformly
distributed on $\{0,1,\ldots,q-1\}$. For
$w=(w_1,\ldots,w_n)\in\mathbb{R}^n$, we set
$
S_w=\sum_{i=1}^n w_iX_i.
$
It would be interesting to determine whether the coupling argument developed
in this paper has a $q$-ary analogue and, if so, whether it yields
dimension-free comparisons between the R\'enyi entropies of $S_w$.
\end{enumerate}

{\bf Acknowledgement.} Conjecture \ref{conj:main} was communicated to us by Mokshay Madiman. We thank Zhen Fu, Mokshay Madiman, and Yetong Sha for valuable discussions. This work was supported by the National Natural Science Foundation of China (NSFC) grant no. 62201175.


\bibliographystyle{plain}

\end{document}